\documentclass[12pt,reqno]{amsart}
\usepackage{amsmath, amsthm, amssymb}

\usepackage{color}
\usepackage{amscd}
\usepackage{xcolor}

\advance \topmargin by -\headheight
\advance \topmargin by -\headsep
     
\evensidemargin \oddsidemargin
\newtheorem{theorem}{Theorem}[section]
\newtheorem{lemma}[theorem]{Lemma}
\newtheorem{corollary}[theorem]{Corollary}

\newtheorem{proposition}[theorem]{Proposition}

\theoremstyle{definition}
\newtheorem{definition}[theorem]{Definition}

\theoremstyle{remark}
\newtheorem{remark}[theorem]{Remark}

\numberwithin{equation}{section}

\newcommand{\imod}[1]{\,\,\text{mod}\,\,#1}

\def\calB{{\mathcal B}} 
\def\calC{{\mathcal C}} 

\def\calD{{\mathcal D}}
\def\calE{{\mathcal E}}

\def\calM{{\mathcal M}}
\def\calN{{\mathcal N}}
\def\calO{{\mathcal O}}

\def\calQ{{\mathcal Q}}
\def\calR{{\mathcal R}} 
\def\calS{{\mathcal S}}

\def\F{{\mathbb F}}\def\N{{\mathbb N}}
\def\R{{\mathbb R}}
\def\Z{{\mathbb Z}}\def\Q{{\mathbb Q}}

\def\alp{{\alpha}}

\def\del{{\delta}}

\def\tet{{\theta}}  

\def\lam{{\lambda}} \def\Lam{{\Lambda}}

\def\d{{\partial}}
\def\eps{\varepsilon}

\def\le{\leqslant} \def\ge{\geqslant}

\def\d{{\,{\rm d}}}

\makeatletter
\def\imod#1{\allowbreak\mkern5mu({\operator@font mod}\,\,#1)}

\DeclareMathOperator{\lcm}{lcm}
\DeclareMathOperator{\sgn}{sgn}

\begin{document}

\author{Peter Cho-Ho Lam}

\address{Department of Mathematics\\
Simon Fraser University\\
Burnaby, BC V5A 1S6\\
CANADA}
\email{chohol@sfu.ca}
\urladdr{http://www.sfu.ca/~chohol/}

\author{Damaris Schindler}

\address{Mathematisch Instituut\\ Universiteit Utrecht\\ Budapestlaan~6\\ NL-3584 CD~Utrecht\\ The Netherlands}
\email{d.schindler@uu.nl}
\urladdr{http://www.uu.nl/staff/DSchindler}

\author{Stanley Yao Xiao}

\address{Mathematical Institute\\
University of Oxford\\
Andrew Wiles Building\\
Radcliffe Observatory Quarter\\
Woodstock Road\\
Oxford\\
OX2 6GG\\UK}
\email{stanley.xiao@maths.ox.ac.uk}
\urladdr{https://www.maths.ox.ac.uk/people/stanley.xiao}


\title[On prime values of binary quadratic forms with a thin variable]{On prime values of binary quadratic forms with a thin variable}

\date{\today}

\begin{abstract} In this paper we generalize the result of Fouvry and Iwaniec dealing with prime values of the quadratic form $x^2 + y^2$ with one input restricted to a thin subset of the integers. 
We prove the same result with an arbitrary primitive positive definite binary quadratic form.
In particular, for any positive definite binary quadratic form $F$ and binary linear form $G$, there exist infinitely many $\ell, m\in\Z$ such that both $F(\ell, m)$ and $G(\ell, m)$ are primes as long as there are no local obstructions.
\end{abstract}

\maketitle


\section{Introduction}

One of the most difficult problems in number theory concerns finding primes among interesting subsets of the natural numbers. A particular example of such a problem is finding primes among values of a given polynomial. Several famous conjectures belong to this line of investigation, including the Bateman-Horn conjecture.\\

In the case of polynomials of a single variable it is unknown whether a given polynomial represents infinitely many primes, except for linear polynomials by the seminal work of Dirichlet. For polynomials in two variables we have some non-linear examples, including quadratic norm forms, all suitable quadratic polynomials by work of Iwaniec \cite{Iwa}, binary cubic forms by work of Heath-Brown \cite{HB1} and Heath-Brown and Moroz \cite{HBM}, and the polynomial $x^2+y^4$ due to Friedlander and Iwaniec \cite{Frie-I}.
One obvious approach is to deduce the analogous results for single variable polynomials from their two-variable counterparts by restricting one variable. Currently we do not know how to do this, but  in some cases we can restrict one of the variables to a sparse subset of the integers. This gives rise to an interesting family of problems.\\

One particular example that has been considered is the case of the quadratic form $x^2+y^2$, where $y$ is restricted to a sparse subset of the integers, including the case of $y$ being prime. This was worked out in great detail by Fouvry and Iwaniec \cite{FI}. Lam \cite{Lam1}\cite{Lam2} and Pandey \cite{Pan} studied similar problems for principal forms of certain negative discriminants. In all these cases the number of admissible $y$ up to size $X$ cannot be less than $O_\del(X^{1-\del})$ for any positive $\del$. Friedlander and Iwaniec \cite{Frie-I} were able to break this barrier in the special case of restricting $y$ to the set of squares. Heath-Brown and Li \cite{HBL} later refined their methods to restrict $y$ to the set of prime squares.\\

In this paper we further generalize the work of Fouvry and Iwaniec \cite{FI} by considering arbitrary primitive positive definite binary quadratic forms. It is worth mentioning that Friedlander and Iwaniec \cite{Frie-I3} provided a simplified proof of \cite{FI} if $y$ is restricted to the set of primes.This was followed by Lam \cite{Lam2} and Pandey \cite{Pan} in their own works but we decided to follow the original argument. \par
Let $F(x,y)\in \Z[x,y]$ be a positive definite and primitive quadratic form (i.e. the greatest common divisor of its coefficients is equal to $1$). For $d\in \mathbb{N}$, we set
$$\rho(d)=\sharp\{\nu \imod d: F(1,\nu)\equiv 0 \imod d\}.$$
We shall prove the following theorem. 

\begin{theorem}\label{thm1}
Let $F(x,y)=\alpha x^2+\beta xy+\gamma y^2\in \Z[x,y]$ be a primitive positive definite quadratic form and $X$ be a positive real number. Let $\lam (\ell)$ be a sequence of complex numbers supported on the natural numbers which satisfy the bound $|\lam(\ell)|\leq C  \log^A\ell $ for all $\ell \in \mathbb{N}$ and some fixed $A, C>0$. Suppose $q\in\N$ and $q\le (\log X)^N$ for some $N>0$. Then for any $B>0$ and $a\in\Z$ with $\gcd(a, q)=1$ we have
$$\sum_{\substack{F(\ell,m)\leq X\\F(\ell, m)\equiv a\imod q}}\lam(\ell)\Lam(F(\ell,m)) = H_{F, q}\sum_{\substack{F(\ell,m)\leq X\\\gcd(\ell, \gamma m)=1\\\gcd(F(\ell, m), P_F)=1\\F(\ell, m)\equiv a\imod q}} \lam(\ell)+ O_{A,B,C,F, N} (X(\log X)^{-B}),$$
where $\Lam$ is the von Mangoldt function and
$$H_{F, q}=\prod_{p\nmid qP_F}\left(1-\frac{\rho(p)}{p}\right)\left(1-\frac{1}{p}\right)^{-1}\prod_{p| qP_F}\left(1-\frac{1}{p}\right)^{-1}.$$
and $P_F = \prod_{p \leq C_F} p$ with $C_F$ depending only on $F$.
\end{theorem}

Note that $H_{F, q}$ is positive if $\rho(2)\neq 2$. For example the primitive positive definite binary quadratic form $2x^2+xy+y^2$ of discriminant $-7$ cannot represent infinitely many prime values with $x$ a prime. On the other hand it exhibits infinitely many prime values with $y$ prime. The flexibility provided by $\lambda(\ell)$ and $F(\ell, m)\equiv a\imod q$ have applications in proving Vinogradov's three primes theorem with special types of primes; see \cite{Gri} for details.\\

The purpose of introducing $P_F$ in our expression is to remove some small prime factors that avoid the use of Dirichlet composition law (see Section 2). In practice this dependence on $P_F$ can be removed via M\"obius inversion. For example, a particularly attractive consequence of Theorem \ref{thm1} is the following:
\begin{corollary}\label{cor1}
Let $F$ be a positive definite binary quadratic form and $G$ be a binary linear form. Assume that for every prime $p$ there are $x,y\in \Z$ such that $p\nmid F(x, y)G(x,y)$. Then there exist infinitely many $\ell, m\in\mathbb{Z}$ such that both $F(\ell, m)$ and $G(\ell, m)$ are primes.
\end{corollary}
It is also possible to impose the conditions $F(\ell, m)\equiv a\imod q$ and $G(\ell, m)\equiv b\imod q$:

\begin{corollary}\label{cor2}
Let $F(x,y)\in \Z[x,y]$ be a primitive positive definite quadratic form of discriminant $-\Delta$. Suppose $q\in\N$ and $q\le (\log X)^N$ for some $N>0$. Then for any $A>0$ and $a, b\in\Z$ with $\gcd(ab, q)=1$ we have
\begin{gather*}
\sum_{\substack{F(\ell,m)\leq X\\F(\ell, m)\equiv a\imod q\\\ell\equiv b\imod q}}\Lam(\ell)\Lam(F(\ell,m)) =  \frac{H_q\rho(q; a, b)}{q\phi(q)}\frac{\pi X}{\sqrt{|\Delta|}}+ O_{A, N} (X(\log X)^{-A})
\end{gather*}
where
\[
H_q=\prod_{p\nmid q}\left(1-\frac{\rho(p)}{p}\right)\left(1-\frac{1}{p}\right)^{-1}
\]
and
\[
\rho(d; a, b)=\sharp\{\nu \imod d: F(b,\nu)\equiv a \imod d\}.
\]
\end{corollary}

One way to phrase Corollary \ref{cor1} is that given the complete norm form $F(x,y)$ and restricting the first variable $x$ to primes, the form still represents infinitely many primes. A natural extension of this question is to ask given an arbitrary primitive complete norm form $\calN$ in $n$ variables and restricting a subset of the variables to a special set $\calS$, does $\calN$ still represent infinitely many primes? In this formulation Heath-Brown \cite{HB1} and Heath-Brown and Moroz \cite{HBM} can be viewed as restricting one variable in a complete cubic norm form to be equal to zero. More recently, Maynard \cite{May} showed that complete norm forms still represent infinitely many primes even with as many as a quarter of the variables are set to zero.\\

More generally, one expects a polynomial $G$ with exactly $r$ factors over $\Q$ should take values which have exactly $r$ prime factors if there are no local obstructions. Indeed this is included in Schinzel's hypothesis. Our corollary \ref{cor1} is a step towards confirming this conjecture for binary cubic forms, following the theorem of Heath-Brown and Moroz \cite{HBM}, by confirming this for the case when $F$ has one linear factor and negative discriminant.\\

\begin{corollary}\label{cor3}
Let $H(x,y)\in \Z[x,y]$ be a binary cubic form with negative discriminant, that is reducible over $\Q$. Assume that for every prime $p$ there are $x,y\in \Z$ such that $p\nmid H(x,y)$. Then there exist infinitely many integers $x,y$ such that $H(x,y)$ has exactly two prime factors.
\end{corollary}

In particular, there are infinitely many integers with exactly two prime factors that are sums of two cubes, see also work of Pandey \cite{Pan}.\\

To deduce Theorem \ref{thm1} from the work of Fouvry and Iwaniec \cite{FI} we must overcome two difficulties. The first is that the proof of a key lemma which is critical in Fouvry and Iwaniec \cite{FI} fails for a general binary quadratic form. In particular they obtained an optimal spacing result of roots modulo $d$ of the congruence $\nu^2+1\equiv 0 $ modulo $d$. Fortunately an analogous result was developed by Balog, Blomer, Dartyge and Tenenbaum \cite{BBDT}. We will then mimic the argument from \cite{Frie-I2} to finish up the proof in Section \ref{sec5} as the original argument in \cite{FI} is not sufficient for our case. The second issue is that in general the arithmetic over a ring of integers $\calO_K$ with $K$ a quadratic number field is not analogous to the arithmetic over $\Z[i]$ when $\calO_K$ has a non-trivial class group. To overcome this issue we require several applications of the Dirichlet composition law. We will develop the necessary tools in Section \ref{sec2} and then employ it in Section \ref{bi sums}.\\

{\bf Acknowledgements:} We thank Trevor Wooley for pointing out to us the statement in Corollary \ref{cor3}. The second author is supported by a NWO grant 016.Veni.173.016.\\

{\bf Notation:} We write $\tau(n)$ for the divisor function of a natural number $n$. ${\sum}^\flat$ denotes a sum over positive squarefree integers.

\section{Dirichlet composition}\label{sec2}
Let $F(x, y)=\alpha x^2+\beta xy+\gamma y^2$ be a primitive positive definite  quadratic form of discriminant $-\Delta$. For a sequence of complex numbers $\lam(\ell)$, $\ell\in\N$ and $N\in \N$ we define a sequence
\begin{equation}\label{defa}
a_N= \sum_{\substack{F(\ell,m)=N\\\gcd(\ell, \gamma m)=1}}\lam (\ell).
\end{equation}
As $F(x,y)$ is assumed to be positive definite, this is a finite sum.

In \cite{FI}, a key component of the bilinear sum estimates is the identity
\begin{equation} \label{bilinear 1} a_{mn} = \frac{1}{4} \sum_{|w|^2 = m} \sum_{|z|^2 = n} \lambda(\ell),\end{equation}
where $\ell = \Re(\overline{w} z)$, when $F(x,y)=x^2+y^2$ and $\gcd(m, n)=1$. This is based on the classical identity
\[(a^2 + b^2)(c^2 + d^2) = (ac - bd)^2 + (bc + ad)^2\]
and the fact that there is one binary quadratic form of discriminant $-4$ up to (proper) equivalence. We now extend this identity to the case when the class number is not equal to one. To generalize (\ref{bilinear 1}), we will use the \emph{Dirichlet composition law}.


\begin{definition}[Dirichlet composition] \label{Dirichlet comp} Let $f(x, y)=ax^2+bxy+cy^2$ and $F(x, y)=\alpha x^2+\beta xy+\gamma y^2$ be primitive positive definite forms of discriminant $-\Delta<0$ which satisfy $\gcd(a, \alpha, (b+\beta)/2)=1$. Then the \textit{Dirichlet composition} of $f(x, y)$ and $F(x, y)$ is the form
\[
h(x, y)=a\alpha x^2+Bxy+\frac{B^2+\Delta}{4a\alpha}y^2
\]
where $B$ is any integer such that
\[
\begin{split}
B&\equiv b\imod{2a}\\
B&\equiv \beta\imod{2\alpha}\\
B^2+\Delta&\equiv 0\imod{4a\alpha}.
\end{split}
\] 
\end{definition} 

See \cite{Cox} for a good reference in Dirichlet composition. Note that $(b+\beta)/2\in\Z$ since $b^2\equiv\beta^2\equiv-\Delta\imod4$. This composition makes the equivalence class of binary quadratic of discriminant $\Delta$ into an abelian group. The term \textit{composition} is justified by the following identity:
\begin{equation}
\label{comp law}
(au^2+buv+cv^2)(\alpha X^2+\beta XY+\gamma Y^2)=a\alpha W^2+BWZ+\frac{B^2+\Delta}{4a\alpha}Z^2
\end{equation}
where
\begin{equation}\label{mw}
W=\bigg(u-\frac{B-b}{2a}v\bigg)X-\bigg(\frac{B-\beta}{2\alpha}u+\frac{(b+\beta)B+\Delta-b\beta}{4a\alpha}v\bigg)Y
\end{equation}
and
\begin{equation}\label{mz}
Z=\alpha vX+\bigg(au+\frac{b+\beta}{2}v\bigg)Y.
\end{equation}
It is convenient to have explicit coefficients in the composition for our purposes. \\

To establish an analogue of \eqref{bilinear 1}, we need to study the solutions of
\[
mn=F(X, Y)
\]
when $\gcd(m, n)=1$. One can show that $m$ can be represented by a binary quadratic form of the same discriminant, say $f(x, y)$; and by composing with $F(x, y)$ we obtain a form that represents $n$. But to work out the composition explicitly, the condition $\gcd(a, \alpha, (b+\beta)/2)=1$ is needed. This motivates us to construct a set of binary quadratic forms, $\mathcal{S}_F(t)$, in which this condition is always satisfied.\\


For any $F(x, y)=\alpha x^2+\beta xy+\gamma y^2$ and any $t\in\Z$, define $\mathcal{S}_F(t)$ to be a set of binary quadratic form of discriminant $-\Delta$ such that
\begin{enumerate}
\item every primitive binary quadratic form of discriminant $-\Delta$ is properly equivalent to exactly one element in $\mathcal{S}_F(t)$;
\item the principal form is contained in $\mathcal{S}_F(t)$; and
\item the set $\{f(1, 0): f\in\mathcal{S}_F(t)\}$ consists of distinct primes that do not divide $t$.
\end{enumerate}
If $\mathcal{S}_F(t)$ is the set of primitive reduced forms of discriminant $-\Delta$, then (1) and (2) are satisfied. Since each of them represent infinitely primes, if necessary, we can transform the form so that the coefficient of $x^2$ is one of these primes and thus it is clear that (3) can be satisfied. Now we put $\calS_F=\calS_F(\alpha)$ and define
\begin{equation}
\label{PF}
Q_F=2\alpha\gamma\Delta\prod_{f\in\mathcal{S}_F}f(1, 0).
\end{equation}
We assume $P_F$ is large enough so that $Q_F|P_F$. We also pick an integer $B$ with the following properties:
\begin{enumerate}
\item $B\equiv b\imod {2a}$ for all $ax^2+bxy+cy^2\in \mathcal{S}_F$;
\item $B\equiv \beta\imod {2\alpha}$; and
\item $B^2+\Delta\equiv0\imod {4a\alpha}$ for all $ax^2+bxy+cy^2\in \mathcal{S}_F$.
\end{enumerate}
So $B$ only depends on $F$ and the choice of $\mathcal{S}_F$.\\


\begin{proposition}
\label{bqf1}
Let $\Delta$ be a positive integer and $F(x, y)=\alpha x^2+\beta xy+\gamma y^2$ be a primitive binary quadratic form of discriminant $-\Delta$. Let $m, n$ be positive integers such that $\gcd(mn, P_F)=1$. If $mn=F(X, Y)$ for some integers $X, Y$ with $\gcd(X, Y)=1$, then there exists a unique binary quadratic form $f(x,y)=ax^2+bxy+cy^2\in \mathcal{S}_F$ and integers $u, v, w, z$ such that $\gcd(u,v)=\gcd(w,z)=1$ and
\begin{equation*}
\begin{split}
au^2+buv+cv^2&=m,\\
a\alpha w^2+Bwz+\frac{B^2+\Delta}{4a\alpha}z^2&=n,\\
\bigg(au+\frac{b+\beta}{2}v\bigg)w+\bigg(\frac{B-\beta}{2\alpha}u+\frac{(b+\beta)B+\Delta-b\beta}{4a\alpha}v\bigg)z&=X,\\
-\alpha vw+\bigg(u-\frac{B-b}{2a}v\bigg)z&=Y;
\end{split}
\end{equation*}
and if $\Delta>4$ then there is exactly one more tuple, namely $(-u, -v, -w, -z)$, that satisfies the properties. If $\Delta=3, 4$ we have $6$ or $4$ solutions, respectively.
\end{proposition}

\begin{proof}Choose an integer $\nu$ such that $\nu\equiv (2\alpha X+\beta Y)Y^{-1}\imod m$. Then $4m|\nu^2+\Delta$ and we define the primitive binary quadratic form
\[
M(x,y)=mx^2+\nu xy+\frac{\nu^2+\Delta }{4m}y^2.
\]
Thus $M(x,y)$ is properly equivalent to some $f(x,y)=ax^2+bxy+cy^2\in\mathcal{S}_F$. By construction, there exist integers $u, v, r, s$ such that $us-rv=1$ and
\begin{equation}
\label{equivalence}
M(x,y)=f(ux+ry, vx+sy).
\end{equation}
Therefore, $m=M(1,0)=f(u,v)=au^2+buv+cv^2$. By comparing the coefficients of $xy$ in \eqref{equivalence}, we deduce that
\[
\nu=2aur+bus+brv+2cvs.
\]
Consequently, we have $\nu v\equiv-(2au+bv)\imod m$ and therefore
\begin{equation}
\label{cong1}
(2\alpha X+\beta Y)v\equiv -(2au+bv)Y\imod m.
\end{equation}
By $(\nu^2+\Delta)vY \equiv0\imod m$, we obtain
\begin{equation}
\label{cong2}
(2au+bv)(2\alpha X+\beta Y)-\Delta vY\equiv0\imod m.
\end{equation}
Now define $W, Z$ as in \eqref{mw} and \eqref{mz}. By \eqref{cong1} and \eqref{cong2}, they are both divisible by $m$. Take $w=W/m, z=Z/m$. Then $n=a\alpha w^2+Bwz+\frac{B^2+\Delta}{4a\alpha}z^2$. Solving \eqref{mw} and \eqref{mz} gives
\[
\begin{split}
\bigg(au+\frac{b+\beta}{2}v\bigg)w+\bigg(\frac{B-\beta}{2\alpha}u+\frac{(b+\beta)B+\Delta-b\beta}{4a\alpha}\bigg)z&=X,\\
-\alpha vw+\bigg(u-\frac{B-b}{2a}v\bigg)z&=Y.
\end{split}
\]
These equations also imply that $\gcd(u,v)|\gcd(X,Y)$ and $\gcd(z,w)|\gcd(X,Y)$, hence $\gcd(u,v)=\gcd(z,w)=1$.\\

The choice of $f(x, y)$ is unique since $f(x, y)$ is properly equivalent to
\[
M(x,y)=mx^2+\nu xy+\frac{\nu^2+\Delta }{4m}y^2,
\]
and it can easily be checked that for any $M'(x,y)$ constructed with a different $\nu' \equiv \nu \pmod{m}$ and $(\nu')^2 + \Delta \equiv 0 \pmod{4m}$, $M,M'$ are properly equivalent. \\

Now suppose $\Delta>4$ and there is another tuple $(u_0, v_0, w_0, z_0)$ that satisfies the requirement. It is straightforward to verify that
\begin{equation}
\label{am}
16=\bigg(\frac{(2au+bv)(2au_0+bv_0)+\Delta vv_0}{am}\bigg)^2+\Delta \bigg(\frac{2(uv_0-u_0v)}{m}\bigg)^2.
\end{equation}
Further, it is easy to see that $uv_0 - u_0 v \equiv 0 \pmod{m}$. It thus follows that 
\[4 \Delta > 16 \geq \Delta \left(\frac{2(uv_0 - u_0 v)}{m}\right)^2,\] 
whence $uv_0 - v_0 u = 0$. It then follows that $(u_0, v_0, w_0, z_0) = \pm (u,v,w,z)$. \\

If $\Delta=4$, we can take $\mathcal{S}_F=\{x^2+y^2\}$ and $B=\beta$ (note that $\beta$ must be even). Then \eqref{am} becomes
\[
1=\bigg(\frac{uu_0+ vv_0}{m}\bigg)^2+\bigg(\frac{uv_0-u_0v}{m}\bigg)^2
\]
and this has 4 pairs of solutions $(u_0, v_0)$. The case for $\Delta=3$ is similar.
\end{proof}
On the other hand, if we have $f(u, v)=m$ and $f^*(w, z)=n$, by Dirichlet composition they can produce $X, Y\in\Z$ such that $F(X, Y)=mn$ via
\[
\begin{split}
\bigg(au+\frac{b+\beta}{2}v\bigg)w+\bigg(\frac{B-\beta}{2\alpha}u+\frac{(b+\beta)B+\Delta-b\beta}{4a\alpha}v\bigg)z&=X,\\
-\alpha vw+\bigg(u-\frac{B-b}{2a}v\bigg)z&=Y.
\end{split}
\]
However even if $\gcd(u, v)=\gcd(w, z)=1$, it does not guarantee $\gcd(X, Y)=1$. We are not too far away because by \eqref{mw} and \eqref{mz}, we have
\[\gcd(X, Y)|\gcd(mw, mz)=m.\]
Similarly $\gcd(X, Y)|n$. Hence if $\gcd(m, n)=1$, we have $\gcd(X, Y)=1$. Furthermore, if $\gcd(mn, P_F)=1$, we also deduce that $\gcd(X, \gamma)=1$ since $\gamma|P_F$. From Proposition \ref{bqf1} and the discussion above, we conclude that
\begin{proposition}
\label{bqf2}
If $\gcd(m, n)=\gcd(mn, P_F)=1$, we have
\begin{equation}
\label{adecomp}
a_{mn}=\frac{1}{2}\sum_{f\in\mathcal{S}_F}\sum_{\substack{(w, z)\in\mathbb{Z}^2\\f^*(w, z)=m\\(w, z)=1}}\sum_{\substack{(u,v)\in\mathbb{Z}^2\\f(u,v)=n\\(u, v)=1}}\lam(\calQ_F(u, v; w, z))
\end{equation}
where
\[
f^*(w, z)=a\alpha w^2+Bwz+\frac{B^2+\Delta}{4a\alpha}z^2
\]
and
\[
\calQ_F(u, v; w, z)=\alpha vw+\bigg(u-\frac{B-b}{2a}v\bigg)z.
\]
Here we set $\lam(\ell)=0$ if $\ell<0$.  If $\Delta=3$ or $4$, the constant $\frac{1}{2}$ before the summation should be $\frac{1}{6}$ and $\frac{1}{4}$ respectively.
\end{proposition}
The condition $\gcd(mn, P_F)=1$ also implies
\begin{equation}
\label{uBv}
\gcd(\calQ_F(u, v; w, z), \alpha)=1
\end{equation}
and we will need this later in Section \ref{bi sums}.

\section{Setting up a sieve problem}

After the algebraic preparations we now present the general framework of sieving with which we aim to find prime values in the sequence $F(\ell,m)$ with $\ell$ restricted to a thin sequence. In order to prove Theorem \ref{thm1} it suffices to consider the sum
\begin{equation} \label{Pchi} P(X; \chi)=\sum_{\substack{N\leq X\\\gcd(N, P_F)=1}}a_N\chi(N)\Lam(N),\end{equation}
where $\Lam(N)$ is the von Mangoldt function and $a_N$ is defined as in (\ref{defa}) and $\chi$ is a Dirichlet character modulo $q$. The character $\chi$ is present to detect the congruence condition modulo $q$. \\

Let $Y,Z>1$ be such that $X>YZ$. Put
\begin{equation} \label{bxyz} B(X;Y,Z;\chi):= \sum_{\substack{bd\leq X\\ b>Y\\\gcd(bd, P_F)=1}}\mu(b)\left(\sum_{\substack{c\mid d\\ c>Z}}\Lam(c)\right)a_{bd}\chi(bd),\end{equation} 
and 
$$\del(N;Y,Z)=\sum_{b>Y}\frac{\mu(b)}{b}\left\{\rho(b) \log \frac{N}{b}-\sum_{c\leq Z}\frac{\Lam(c)}{c}\rho(bc)\right\}.$$

Then we have the analogue of Proposition 9 in \cite{FI}.

\begin{proposition}\label{prop9}
Let $Y,Z\geq 1$ and $X>YZ$. Then we have the identity
\begin{equation*}
\begin{split}
P(X; \chi)=&\sum_{\substack{N\leq X\\\gcd(N, P_F)=1}}\sum_\ell \lam(\ell;N)(H_{F, q}+\del(N;Y,Z))\\ &+B(X;Y,Z;\chi)+R(X;Y,Z;\chi)+P(Z;\chi)
\end{split}
\end{equation*}
where
\[
H_{F, q}=\prod_{p\nmid qP_F}\left(1-\frac{\rho(p)}{p}\right)\left(1-\frac{1}{p}\right)^{-1}
\]
and $R(X;Y,Z,\chi)$ is given by (\ref{RT1}). 
\end{proposition}

From Proposition \ref{prop9} we see that Theorem \ref{thm1} follows provided that acceptable estimates for $\delta(N; Y, Z), B(X;Y, Z, \chi), R(X; Y, Z, \chi), P(Z; \chi)$ can be obtained. We will give appropriate bounds for all but $B(X;Y,Z, \chi)$ in this section. \\

When sieving for prime values of $N$ we will need to study sums of the type
\begin{equation} \label{Ad} A_d(X; \chi):= \sum_{\substack{N\leq X\\ N\equiv 0 \imod d\\\gcd(N, P_F)=1}}a_N\chi(N) \end{equation}
for $d$ a positive integer. Note that $A_d(X; \chi)=0$ if $\gcd(d, qP_F)>1$. If $N=F(\ell, m)\equiv0\imod d$ then from $\gcd(\ell, \gamma m)=1$ we immediately have $\gcd(\ell, d)=1$ as well. Hence we expect that $A_d(X; \chi)$ is approximated by
\begin{align*} M_d(X; \chi) & = \frac{\rho(d)}{d}\sum_{\substack{N\le X\\\gcd(N, P_F)=1}}\sum_{\gcd(\ell, d)=1}\lam(\ell; N) \\ 
& = \frac{\rho(d)}{d}\mathop{\sum\sum}_{\substack{F(\ell, m)\le X\\\gcd(\ell, \gamma md)=1\\\gcd(F(\ell, m), P_F)=1}}\lam(\ell)\chi(F(\ell, m))\end{align*}
when $\gcd(d, qP_F)=1$, where
$$\lam(\ell;N)=\chi(N)\sum_{\substack{m\in\mathbb{Z}\\F(\ell,m)=N\\\gcd(\ell, \gamma m)=1}}\lam(\ell);$$
and $M_d(X; \chi)=0$ otherwise. With this we set
\begin{equation}\label{defR}
R_d(X; \chi)=A_d(X; \chi)-M_d(X; \chi).
\end{equation}
For a parameter $D$ we define the complete remainder term $R(X,D; \chi)$ as
\begin{equation}\label{Alem main}R(X,D; \chi)=\sum_{d\leq D}|R_d(X; \chi)|.\end{equation}
As $F(x,y)$ is assumed to be positive definite, there exists a positive constant $C_1$ only depending on $F$, such that $F(\ell,m)\leq X$ implies that $|\ell|, |m|\leq C_1\sqrt{X}$. Now put

\begin{equation} \label{RT1} R(X;Y,Z;\chi)=\sum_{b\leq Y}\mu(b)\left\{ R_b(X; \chi)\log\frac{X}{b}-\int_1^XR_b(t; \chi)\frac{\d t}{t}-\sum_{c\leq Z}\Lam(c)R_{bc}(X; \chi)\right\}.\end{equation}
As in \cite{FI} we have the bound
\begin{equation} \label{Alem bd} |R(X;Y,Z;\chi)|\leq R(X,YZ; \chi) \log X +\int_1^X R(t,Y; \chi) \frac{\d t}{t}.\end{equation}

\begin{proof}[Proof of Proposition \ref{prop9}] Our goal is to derive Proposition \ref{prop9} from \cite[Proposition 9]{FI}. To do so we must check that that condition (7.16) in \cite{FI} holds with the function $\rho(bc)$, with $c$ a fixed integer. Let $-D(l)$ denote the discriminant of the quadratic polynomial $F(x,l)$, and let $-d = -d(l)$ be the unique fundamental discriminant such that $\Q(\sqrt{-D(l)}) = \Q(\sqrt{-d(l)})$. Let $\rho'_d(n)$ be the multiplicative function 
\[\rho'_{-d}(n) = \sum_{m | n} \left(\frac{-d}{m}\right).\]
Consider the Dirichlet series
\[\mathfrak{D}(s) = \sum_{n=1}^\infty \frac{\mu(n) \rho(cn)}{n^s}. \]
Then $\mathfrak{D}(s)$ differs from the series
\[\calE(s) = \sum_{n=1}^\infty \frac{\mu(n) \rho'_{-d}(n)}{n^s}  =\prod_p \left(1 - \frac{\rho'_{-d}(p)}{p^s}\right) \]
by a holomorphic factor. Here
\[\rho'_{-d}(p) = \begin{cases} 1 & \text{if } p | d \\ \\ 2 & \text{if } \left(\frac{-d}{p}\right) = 1 \\ \\ 0 & \text{if } \left(\frac{-d}{p}\right) = -1. \end{cases}\]
It is then apparent that 
\[\calE(s) = \frac{\mathcal{G}(s)}{\zeta(s) L(s, \chi_{-d})},\]
where
\[\mathcal{G}(s) = \prod_{p : \chi_{-d}(p) = 1} \left(1 + \frac{1}{p^{2s} - 2p^s} \right)^{-1} \prod_{p : \chi_{-d}(p) = 1} \left(1 - \frac{1}{p^{2s}}\right)^{-1}. \]
Plainly, $\mathcal{G}(s)$ converges and is holomorphic for $\Re(s) > 1/2$. We then obtain the bound 
\[\sum_{n \leq X} \mu(n) \rho'_{-d}(n) = O\left(X \exp\left(-c_d \sqrt{\log X}\right) \right)\]
for some positive number $c_d$ by standard estimates of the zero-free region of the Dirichlet $L$-function $L(s,\chi_{-d})$ and the Selberge-Delange method. The desired conclusion then follows from partial summation. 
\end{proof}

The terms $B(X;Y,Z,\chi), R(X;Y,Z,\chi)$ in Proposition \ref{prop9} will be controlled by the following lemmas: 

\begin{lemma} \label{Alem} Suppose $q\in\N$ and $q\le (\log X)^N$ for some $N>0$. Let $\varepsilon >0$. Assume that $Y,Z>1$ and $YZ<X^{1-\varepsilon}$. Then we have
$$R(X;Y,Z; \chi)\ll_{\varepsilon} X^{1-\varepsilon/5}.$$
\end{lemma}

\begin{lemma}\label{Blem}
Suppose $q\in\N$ and $q\le (\log X)^N$ for some $N>0$. Let $\theta_1, \theta_2$ be two real numbers such that $1/2<\theta_1<1$ and $0<\theta_2<1-\theta_1$. Then for $Y=X^{\theta_1}$ and $Z=X^{\theta_2}$ and any $B>0$,
$$B(X; Y, Z; \chi)\ll X(\log X)^{-B}.$$
\end{lemma}

We follow a similar strategy to \cite{FI} in showing that the term $R(X;Y,Z; \chi)$ can be bounded by our Type I estimate from Lemma \ref{lemR1} and the trivial estimate $R(t, Y; \chi)\ll t^{1+\eps}(\log Y)^2$. \\

With all these ingredients we can prove our main theorem and corollaries.

\begin{proof}[Proof of Theorem \ref{thm1}] 
Define $Y, Z$ as in Lemma \ref{Blem}. Together with Lemma \ref{Alem} and Lemma \ref{Blem} we have shown that
\begin{equation}
\label{apply}
\begin{split}
&\sum_{\substack{F(\ell, m)\leq X\\\gcd(\ell, \gamma m)=1\\\gcd(F(\ell, m), P_F)=1}}\lambda(\ell)\chi(F(\ell, m))\Lam(F(\ell, m))\\
=&~H_{F, q}\sum_{\substack{F(\ell, m)\leq X\\\gcd(\ell, \gamma m)=1\\\gcd(F(\ell, m), P_F)=1}}\lam(\ell)\chi(F(\ell, m))+O_{A, B, C, F, N}(X(\log X)^{-B}).
\end{split}
\end{equation}
The condition $\gcd(F(\ell, m), P_F)=1$ on the left can be removed because of the presence of $\Lam(F(\ell, m))$. Hence by orthogonality of $\chi$, it gives
\[
\sum_{\substack{F(\ell, m)\leq X\\F(\ell, m)\equiv a\imod q}}\lambda(\ell)\Lam(F(\ell, m))=H_{F, q}\sum_{\substack{F(\ell, m)\leq X\\\gcd(\ell, \gamma m)=1\\\gcd(F(\ell, m), P_F)=1\\F(\ell, m)\equiv a\imod q}}\lam(\ell)+O_{A, B, C, F, N}(X(\log X)^{-B}).
\]
Finally, we treat the remaining terms in Proposition \ref{prop9}. We can use the trivial bound $P(Z;\chi)\ll_\eps Z^{1+\epsilon}$ for any $\eps >0$. The contribution of the terms with $\del(N;Y,Z)$ is negligible as in (7.18) in \cite{FI}. \end{proof}

\section{Level of Absolute Distribution}\label{sec5} 

In this section we shall obtain Type I estimates that are needed to prove Lemma \ref{Alem} and the corollaries to Theorem \ref{thm1}. The most pressing issue is to control the quantity $R(X;D,\chi)$ given (\ref{Alem main}). To this end, we have the following lemma: 

\begin{lemma}\label{lemR1}
For $1\leq D\leq X$ we have the bound
$$R(X,D; \chi)\ll_\eps q^3D^{1/4}X^{3/4+\eps}.$$
\end{lemma}

To prove Lemma \ref{lemR1}, it is convenient to remove the restrictive condition $\gcd(N,P_F) = 1$. We let the scripted letters $\mathcal{A}, \mathcal{M}, \mathcal{R}$ to denote the analogous quantities $A,M,R$ which appeared in the previous section, but without the condition $\gcd(N,P_F) = 1$. We also set $\calM_d(X; \chi)=0$ if $\gcd(d, q)>1$. We then have the following analogue to Lemma \ref{lemR1}:

\begin{lemma}\label{lemR2}
For $1\leq D\leq X$ we have the bound
$$\mathcal{R}(X,D; \chi)\ll_\eps q^3D^{1/4}X^{3/4+\eps}.$$
\end{lemma}

Lemma \ref{lemR1} will be a simple consequence of Lemma \ref{lemR2}. Furthermore Lemma \ref{lemR2} will be used to prove the corollaries from our main theorem. As in \cite{FI} we deduce Lemma \ref{lemR2} from a version where the $\mathcal{A}_d(X; \chi)$ are smoothed with an auxiliary weight function. Since we need to accommodate extra assumption $\gcd(\ell, \gamma m)=1$ it is convenient to adopt the approach from \cite{Frie-I2} instead. \\

Let $\sqrt{X}\le Y\le X$ be an additional parameter to be chosen later, and let $w: \R^+\rightarrow \R$ be a smooth function with the following properties:
\begin{equation}
\label{FF}
\left\{
\begin{array}{lr}
w(u)=0\hspace{12mm} \text{if }u\not\in[1, X],\vspace{2mm}\\
0\le w(u)\le 1\hspace{5.5mm}\text{if }u\in[1, X],\vspace{2mm}\\
w(u)=1\hspace{12mm}\text{if }Y\le u\le X-Y,\vspace{2mm}\\
w^{(j)}(u)\ll Y^{-j} \hspace{4mm}\text{for }j=1, 2.
\end{array}
\right.
\end{equation}
For $a, \ell\geq 1$ we define the function
\begin{equation}
\label{Fal}
F_{a, \ell}(z):=\int_\R w(F(a\ell, at))e(-zt)\d t.
\end{equation}
Let
\[
\mathcal{A}_d(X;w, \chi):=\sum_{N\equiv 0 \imod d}a_Nw(N)\chi(N).
\]
When $\gcd(d, q)=1$, define
\begin{equation}
\label{MdXw}
\mathcal{M}_d(X; w, \chi)=\frac{\rho(d)}{d}\sum_{\gcd(\ell, \gamma d)=1}\frac{\lam(\ell)\phi(\ell)}{\ell}\bigg(\frac{\sum_{k\imod q}\chi(F(\ell, k))}{q}\bigg)F_{1, \ell}(0)
\end{equation}
as well as the smoother remainder term
\[
\mathcal{R}_d(X;w, \chi):=\mathcal{A}_d(X;w, \chi)-\mathcal{M}_d(X;w, \chi).
\]
When $\gcd(d, q)>1$, they are both defined to be 0. We obtain the following lemma.

\begin{lemma}\label{lemR3}
Let $w$ and $\lam$ be as above and $1\leq D\leq X$. Then one has
$$\sum_{d\leq D}|\mathcal{R}_d(X;w, \chi)|\ll_\eps \frac{q^3D^{1/2}X^{3/2+\epsilon}}{Y}.$$
\end{lemma}

\begin{proof}
Note that
\[
\mathcal{A}_d(X; w, \chi)=\sum_{N\equiv0\imod d}\chi(N)w(N)\mathop{\sum\sum}_{\substack{F(\ell, m)=N\\\gcd(\ell, \gamma m)=1}}\lambda(\ell).
\]
We will assume $\gcd(d, q)=1$ throughout. The conditions $\gcd(\ell, \gamma m)=1$ and $F(\ell, m)\equiv 0\imod d$ imply $\gcd(\ell, d)=1$; hence $\mathcal{A}_d(X; w, \chi)$ can be rewritten as
\[
\mathcal{A}_d(X; w, \chi)=\sum_{\gcd(\ell, \gamma)=1}\lam(\ell)\sum_{\substack{\nu\imod d\\F(1, \nu)\equiv0\imod d}}\sum_{\substack{m\equiv \ell\nu\imod d\\\gcd(\ell, m)=1}}\chi(F(\ell, m))w(F(\ell, m)).
\]
By M\"obius inversion, we can trade the condition $\gcd(\ell, m)=1$ with
\[
\sum_{\gcd(a, \gamma)=1}\mu(a)\sum_{\gcd(\ell, \gamma)=1}\lam(a\ell)\sum_{\substack{\nu\imod d\\F(1, \nu)\equiv0\imod d}}\sum_{am\equiv a\ell\nu\imod d}\chi(F(a\ell, am))w(F(a\ell, am))
\]
and $a$ is bounded by $O(\sqrt{X})$. Now the innermost sum can be rewritten as
\[
\sum_{k\imod q}\chi(F(a\ell, ak))\sum_{\substack{am\equiv a\ell\nu\imod {d}\\ m\equiv k\imod q}}w(F(a\ell, am)).
\]
To simplify our notation, let $c=d/\gcd(a, d)$. Then the condition $am\equiv a\ell\nu\imod d$ is the same as $m\equiv \ell\nu\imod c$. By Poisson summation formula and Chinese Remainder Theorem,
\begin{equation}
\sum_{\substack{m\equiv \ell\nu\imod c\\ m\equiv k\imod q}}w(F(a\ell, am))=\frac{1}{cq}\sum_{h\in\mathbb{Z}}e\bigg(\frac{hk\overline{c}}{q}\bigg)e\bigg(\frac{h\ell\nu\overline{q}}{c}\bigg)F_{a, \ell}\bigg(\frac{h}{cq}\bigg)
\end{equation}
where $F_{a, \ell}(z)$ is defined in \eqref{Fal}. Therefore
\begin{equation}
\begin{split}
\mathcal{A}_d(X; w, \chi)~=~&\frac{1}{q}\sum_{\gcd(a, \gamma)=1}\frac{\mu(a)}{c}\sum_{\gcd(\ell, \gamma)=1}\lam(a\ell)\sum_{\substack{\nu\imod d\\F(1, \nu)\equiv0\imod d}}\sum_{k\imod q}\chi(F(a\ell, ak))\\
&\sum_{h\in\mathbb{Z}}e\bigg(\frac{hk\overline{c}}{q}\bigg)e\bigg(\frac{h\ell\nu\overline{q}}{c}\bigg)F_{a, \ell}\bigg(\frac{h}{cq}\bigg).
\end{split}
\end{equation}
Define $\mathcal{M}_d(X; w, \chi)$ to be the summand when $h=0$, i.e., the expression
\begin{equation}
\label{Mdf}
\frac{\rho(d)}{q}\sum_{\gcd(a, \gamma)=1}\frac{\mu(a)}{ac}\sum_{\gcd(\ell, \gamma)=1}\lam(a\ell)\sum_{k\imod q}\chi(F(a\ell, ak))\int^\infty_{-\infty}w(F(a\ell, t))\, dt
\end{equation}
and $\mathcal{R}_d(X; w, \chi)=\mathcal{A}_d(X; w, \chi)-\mathcal{M}_d(X; w, \chi)$. This is consistent with \eqref{MdXw} since
\[
\sum_{\substack{a|\ell\\\gcd(a, q)=1}}\frac{\mu(a)\gcd(a, d)}{a}=\prod_{p|\ell, p\nmid q}\bigg(1-\frac{\gcd(p, d)}{p}\bigg)
\]
equals to $0$ if $\gcd(\ell, d)>1$ and thus
\[
\mathcal{M}_d(X; w, \chi)=~\frac{\rho(d)}{dq}\sum_{\gcd(\ell, \gamma d)=1}\frac{\lam(\ell)\phi(\ell)}{\ell}\sum_{k\imod q}\chi(F(\ell, k))F_{1, \ell}(0).
\]
Note that for any integer $c$, we have
\begin{equation}
\label{Fm}
F_{a, \ell}\bigg(\frac{h}{cq}\bigg)=\frac{\sqrt{X}}{|h|}\int^{\infty}_{-\infty}w\bigg(a^2(\alpha \ell^2+\frac{\beta\ell u\sqrt{X}}{h}+\frac{\gamma Xu^2}{h^2})\bigg)e\bigg(-\frac{\sqrt{X}u}{cq}\bigg)\, du.
\end{equation}
We wish to sum $\calR_d(X; w, \chi)$ dyadically and hence we define
\[
\mathcal{R}(X, D; w, \chi)=\sum_{D<d\le 2D}|\mathcal{R}_d(X; w, \chi)|.
\]
Substituting $b=d/c=\gcd(a, d)$, each term in the above sum can be bounded by
\[
|\mathcal{R}_d(X; w, \chi)|~\le~ \frac{1}{dq}{\sum_a}^\flat\sum_{\substack{bc=d\\b|a}}\rho(b)b\sum_{\substack{\nu\imod c\\F(1, \nu)\equiv0\imod c}}\sum_{k\imod q}|W_a(c, \nu)|
\]
where
\[
W_a(c, \nu)=\sum_{h\neq0}\sum_{(\ell, \gamma d)=1}\\\lambda(a\ell)\chi(F(a\ell, ak))e\bigg(\frac{hk\overline{c}}{q}\bigg)e\bigg(\frac{h\ell\nu\overline{q}}{c}\bigg)F_{a, \ell}\bigg(\frac{h}{cq}\bigg).
\]
Hence
\begin{equation}
\label{RXDw}
\mathcal{R}(X, D; w, \chi)~\le~\frac{1}{Dq}\sum_{k\imod q}{\sum_a}^\flat\sum_{b|a}\rho(b)bV_a(D/b)
\end{equation}
where
\[
V_a(C)=\sum_{C<c\le 2C}\sum_{\substack{\nu\imod c\\F(1, \nu)\equiv0\imod c}}|W_a(c, \nu)|.
\]
By dyadic division,
\begin{equation}
\label{messup}
V_a(C)\le\sum_{H}\bigg(V^+_a(C, H)+V^-_a(C, H)\bigg)
\end{equation}
where $H$ is a power of 2,
\[
\begin{split}
V^+_a(C, H)=&\sum_{C<c\le2C}\sum_{\substack{\nu\imod c\\F(1, \nu)\equiv0\imod c}}\bigg|\sum_{H\le h<2H}\sum_{\gcd(\ell, \gamma)=1}\chi(F(a\ell, ak))\\
&\lambda(a\ell)e\bigg(\frac{hk\overline{c}}{q}\bigg)e\bigg(\frac{h\ell\nu\overline{q}}{c}\bigg)F_{a, \ell}\bigg(\frac{h}{cq}\bigg)\bigg|
\end{split}
\]
and $V^-_a(C, H)$ is defined similarly for those $h<0$.
We only present the argument for $V^+_a(C, H)$ below for simplicity. For a reduced residue class $t\imod q$, we define
\[
\alpha_{h,\ell, t}(u)=\chi(F(a\ell, ak))\lambda(a\ell)\frac{H}{h}w\bigg(a^2(\alpha \ell^2+\frac{\beta\ell u\sqrt{X}}{h}+\frac{\gamma Xu^2}{h^2})\bigg)e\bigg(\frac{hk\overline{t}}{q}\bigg).
\]
Then
\begin{equation}
\label{R}
\begin{split}
V^+_a(C, H)\ll&\frac{\sqrt{X}}{H}\int^{C_1H/a}_{-C_1H/a}\mathop{{\sum}^*}_{t\imod q}\sum_{C<c\le 2C}\sum_{\substack{\nu\imod c\\F(1, \nu)\equiv0\imod c}}\bigg|\sum_{H\le h<2H}\\
&\sum_{\gcd(\ell, \gamma)=1}\alpha_{h, \ell, t}(u)e\bigg(\frac{h\ell\nu\overline{q}}{c}\bigg)\bigg|\, du.
\end{split}
\end{equation}
The symbol $\sum^*$ means we are summing over reduced residue classes only. Next, we need to employ the Proposition 3 from \cite{BBDT}.
\begin{proposition}
\label{largesieve}
Let $F(x, y)=\alpha x^2+\beta xy+\gamma y^2\in\mathbb{Z}[x, y]$ be an arbitrary quadratic form whose discriminant is not a perfect square. For any sequence $\alpha_n$ of complex numbers, positive real numbers $D, N$, we have
\[
\sum_{D\le d\le 2D}\sum_{F(\nu, 1)\equiv0\imod d}\bigg|\sum_{n\le N}\alpha_{n}e\bigg(\frac{\nu n}{d}\bigg)\bigg|^2\ll_{F}(D+N)\sum_{n}|\alpha_n|^2.
\]
\end{proposition}
Notice that
\[
\sum_h \sum_\ell\alpha_{h, \ell, t}(u)e\bigg(\frac{h\ell\nu\overline{q}}{c}\bigg)=\sum_{0\le h_0, \ell_0< q}\sum_{\substack{h\equiv h_0\imod q\\\ell\equiv \ell_0\imod q}}\alpha_{h, \ell, t}(u)e\bigg(\frac{\nu n}{c}\bigg)e\bigg(\frac{\nu h_0\ell_0\overline{q}}{c}\bigg)
\]
where $n=(h\ell-h_0\ell_0)/q$. Hence for each fixed pair $0\le h_0, \ell_0<q$, we only need to estimate
\[
\sum_{C<c\le 2C}\sum_{F(1, \nu)\equiv0\imod c}\bigg|\sum_{n\le N}\alpha_{n}e\bigg(\frac{\nu n}{c}\bigg)\bigg|
\]
where
\[
\alpha_n=\mathop{\sum\sum}_{\substack{h\equiv h_0\imod q\\\ell\equiv \ell_0\imod q\\\gcd(\ell, \gamma)=1\\h\ell=nq+h_0\ell_0}}\alpha_{h, \ell, t}(u)
\]
and
\[
N=q+\frac{1}{q}(2H)\bigg(\frac{C_1\sqrt{X}}{a}\bigg)\ll\frac{H\sqrt{x}}{a}.
\]
Applying this inequality and Cauchy-Schwarz inequality on \eqref{R}, we deduce that
\[
V^+_a(C, H)\ll\frac{\sqrt{X}}{H}\frac{Hq}{a}\bigg(\bigg(C+\frac{H\sqrt{X}}{a}\bigg)E\bigg)^{1/2}(C\log C)^{1/2}
\]
where
\[
E=\sum_n\bigg(\mathop{\sum\sum}_{\substack{h\equiv h_0\imod q\\\ell\equiv \ell_0\imod q\\\gcd(\ell, \gamma)=1\\h\ell=nq+h_0\ell_0}}|\lambda(a\ell)|\bigg)^2\ll \frac{H\sqrt{X}}{a}\log^{2A+3}(HX).
\]
Hence we obtain
\begin{equation}
\label{R1}
V^+_a(C, H)\ll \frac{X^{3/4}C^{1/2}H^{1/2}}{a^{3/2}}\bigg(C+\frac{H\sqrt{X}}{a}\bigg)^{1/2}\log^{A+2}(HX).
\end{equation}
To develop a similar bound for large values of $H$, we apply integration by parts twice in \eqref{Fm} as in \cite{Frie-I2}, followed with the large sieve type estimate. We arrive at
\begin{equation}
\label{R2}
V^+_a(C, H)\ll \frac{X^{7/4}q^3C^{5/2}a^{1/2}}{H^{3/2}Y^2}\bigg(C+\frac{H\sqrt{X}}{a}\bigg)^{1/2}\log^{A+2}(HX).
\end{equation}
When $H\le aC\sqrt{X}Y^{-1}$, we use \eqref{R1} to deduce that
\[
V^+_a(C, H)\ll \frac{qX^{5/4}CH^{1/2}}{Y^{1/2}a^{3/2}}\log^{A+2}X
\]
and if $H>aC\sqrt{X}Y^{-1}$, we use \eqref{R2} to deduce that
\[
V^+_a(C, H)\ll \frac{q^3X^2C^{5/2}}{Y^2H}\log^{A+2}(HX).
\]
The same estimates hold for $V^-_a(C, H)$ as well. Therefore by \eqref{messup}
\begin{equation}
\label{RfJ}
V_a(C)\le\sum_{H}\bigg(V^+_a(C, H)+V^-_a(C, H)\bigg)\ll~\frac{q^3X^{3/2}C^{3/2}}{aY}\log^{A+2}X
\end{equation}
and by \eqref{RXDw}
\[
\mathcal{R}(X, D; w, \chi)~\le~\frac{q^3X^{3/2}\sqrt{D}\log^{A+2}X}{Y}\sum_{a\le\sqrt{X}}\frac{\tau(a)}{a}~\ll~\frac{q^3X^{3/2}\sqrt{D}\log^{A+4}X}{Y}.
\]
Finally
\[
\sum_{d\le D}|\mathcal{R}_d(X; w, \chi)|\ll\sum_D\mathcal{R}(X, D; w, \chi)\ll\frac{q^3D^{1/2}X^{3/2+\epsilon}}{Y}.
\]
\end{proof}

\begin{proof}[Proof of Lemma \ref{lemR2}]
To complete the proof of Lemma \ref{lemR2}, it suffices to show that the error we made when we replace $\mathcal{A}_d(X; \chi)$ with $\mathcal{A}_d(X; w, \chi)$ is negligible as well, i.e. both $|\mathcal{A}_d(X; \chi)-\mathcal{A}_d(X; w, \chi)|$ and $|\mathcal{M}_d(X; \chi)-\mathcal{M}_d(X; w, \chi)|$ are small. Note that
\[
\sum_{\substack{d\le D\\\gcd(d, q)=1}} |\mathcal{A}_d(X; \chi)-\mathcal{A}_d(X; w, \chi)|\ll (\log X)^A \mathop{\sum\sum}_{X-Y<F(\ell, m)\le X}\tau^2(F(\ell, m))\ll Y\log^{A+3} X.
\]
Here we have used the estimate $a_N\ll_C \tau(N)(\log X)^A$. Similarly,
\[
|\mathcal{M}_d(X; \chi)-\mathcal{M}_d(X; w, \chi)|\ll\frac{\rho(d)}{d}Y\log^{A+1}X.
\]
Summing over $d$ and choosing $Y=D^{1/4}X^{3/4}$, we have
\[
\mathcal{R}(X,D; \chi)=\sum_{d\leq D}|\mathcal{R}_d(X; \chi)|\ll q^3D^{1/4}X^{3/4+\eps}.
\]
\end{proof}
\begin{proof}[Proof of Lemma \ref{lemR1}] Follows from Lemma \ref{lemR2} and Mobius inversion. 
\end{proof}

Finally, we give proofs for the corollaries.

\begin{proof}[Proof of Corollaries \ref{cor1} and \ref{cor2}]
If $\lam$ is supported on primes, then starting from \eqref{apply} again, the right hand side becomes
\[
\begin{split}
\sum_{\substack{F(\ell, m)\leq X\\\gcd(\ell, \gamma m)=1\\\gcd(F(\ell, m), P_F)=1}}\lam(\ell)\chi(F(\ell, m))&=\sum_{e|P_F}\mu(e)\sum_{\substack{F(\ell, m)\leq X\\\gcd(\ell, \gamma m)=1\\F(\ell, m)\equiv0\imod e}}\lam(\ell)\chi(F(\ell, m))\\
&=\sum_{e|P_F}\mu(e)\mathcal{M}_e(X; \chi)+O\bigg(\sum_{e|P_F}|\mathcal{R}_e(X; \chi)|\bigg).
\end{split}
\]
Therefore it is also equal to
\[
\sum_{e|P_F}\mu(e)\frac{\rho(e)}{e}\mathop{\sum\sum}_{\substack{F(\ell, m)\le X\\\gcd(\ell, \gamma me)=1}}\lam(\ell)\chi(F(\ell, m))+O(\mathcal{R}(X, P_F; \chi)).
\]
The contribution when $\gcd(\ell, e)>1$ is negligible. Hence by Lemma \ref{lemR2}
\[
\begin{split}
\sum_{F(\ell, m)\leq X}\lambda(\ell)\chi(F(\ell, m))\Lam(F(\ell, m))=&H_{F, q}\prod_{p|P_F}\bigg(1-\frac{\rho(p)}{p}\bigg)\sum_{\substack{F(\ell, m)\leq X\\\gcd(\ell, \gamma m)=1}}\lam(\ell)\\
&+O_{A, B, C, F, N}(X(\log X)^{-B}).
\end{split}
\]
The contribution when $\gcd(\ell, \gamma m)>1$ is also negligible; therefore by orthogonality,
\[
\sum_{\substack{F(\ell,m)\leq X\\F(\ell, m)\equiv a\imod q}}\lam(\ell)\Lam(F(\ell,m)) =  \frac{H_q\phi(q)}{q}\sum_{\substack{F(\ell,m)\leq X\\F(\ell, m)\equiv a\imod q}}\lam(\ell)
\]
Corollary \ref{cor2} follows by taking $\lam(\ell)=\Lam(\ell)$ when $\ell\equiv b\imod q$. For Corollary \ref{cor1}, let $G(x, y)=mx+ny$ with $\gcd(m, n)=1$. Then there exist integers $s, t$ such that $ms-nt=1$. By a change of variables $u=-tx-sy, v=mx+ny$ we obtain
\[
F(x, y)=F(nu+sv, -mu-tv),
\]
which is a binary quadratic form in $u$ and $v$. The result follows from Corollary \ref{cor2} on the pair of forms $F(nu+sv, -mu-tv)$ and $v$ with $q=1$.
\end{proof}

\section{Bilinear sums}
\label{bi sums}

In this section we shall estimate $B(X;Y,Z; \chi)$ given in (\ref{bxyz}) by proving Lemma \ref{Blem}. For reasons of exposition, we first work under the assumption that $|\lam(\ell)|\le1$ for all $\ell\in\N$. As we save an arbitrary power of $\log X$ in our arguments, the general case can then be obtained by changing the parameter $A$. We proceed as in \cite{FI}. First put $\theta = (\log X)^{-A}$ and write
\begin{equation} \label{BMN} \calB(M,N) = \sum_{\substack{M < m \leq 2M\\\gcd(m, P_F)=1}} \left \lvert \sum_{\substack{N < n < N'\\\gcd(n, P_F)=1}} \mu(n)\chi(n) a_{mn} \right \rvert,\end{equation}
where $N' = e^\theta N$. Using these sums for $M = 2^j Z$ and $N = e^{k \theta} Y$ we get
\begin{equation} |B(X;Y,Z;\chi)| \leq (\log X) \sum_{\substack{\theta x < MN < X \\ M \geq Z, N \geq Y}} \calB(M,N) + O \left(\theta X (\log X)^2\right)
\end{equation}
where the error term $O(\theta X (\log X)^2)$ represents a trivial bound for the contribution of $\mu(b) \chi(bd)a_{bd}$ with $bd \leq 2 \theta X$ or $e^{-2\theta} X < bd \leq X$, which terms are not covered exactly. As in \cite{FI}, we need to show that each short sum $\calB(M,N)$ satisfies
\begin{equation} \calB(M,N) \ll \theta^2 X (\log X)^2.
\end{equation}

Let $\calB_d(M,N)$ denote the sum (\ref{BMN}) restricted to $\gcd(m,n) = d$. We have
\[\calB(M,N) \leq \sum_{d < \theta^{-1}} \calB_d(M,N) + O\left(\theta^2 X \right)\]
where the error term $O(\theta^2 X)$ represents a trivial bound for the contribution of $\mu(n)\chi(mn) a_{mn}$ with $\gcd(m,n) \geq \theta^{-1}$. Note that
\[\calB_d(M,N) \leq \calB_1(dM, N/d).\]
Therefore, the proof of Lemma \ref{Blem} is reduced to showing the estimate
\begin{equation} \calB_1(M,N) \ll \theta^3 X(\log X)^2
\end{equation}
holds for any $M,N$ with $M \geq Z, N \geq \theta Y$ and $\theta X < MN < X$. \\

Define $\alpha(n)=\mu(n)\chi(n)$. When applying Proposition \ref{bqf2} to decompose $a_{mn}$ into solutions of
\[
f(u, v)=m, \hspace{5mm}f^*(w, z)=n,
\]
in fact later in \eqref{ggstar} we will decompose the solutions of $f^*(w, z)=n$ again using the same proposition. We construct $\calS_{f^*}$ in the same way we construct $\mathcal{S}_F$ by taking
\[
\calS_{f^*}=\calS_{f^*}\bigg(\alpha\prod_{f\in\calS_F}f(1, 0)\bigg)
\]
and let $g(x, y)=dx^2+exy+fy^2\in \calS_{f^*}$. We pick an integer $B$ such that
\begin{enumerate}
\item $B\equiv b\imod {2a}$ for all $ax^2+bxy+cy^2\in \mathcal{S}_F$;
\item $B\equiv e\imod {2d}$ for all $dx^2+exy+fy^2\in \mathcal{S}_{f^*}$;
\item $B\equiv \beta\imod {2\alpha}$; and
\item $B^2+\Delta\equiv0\imod {4ad\alpha}$ for all $ax^2+bxy+cy^2\in \mathcal{S}_F$ and $dx^2+exy+fy^2\in \mathcal{S}_{f^*}$.
\end{enumerate}
Such $B$ always exist since the coefficients of $x^2$ of elements in $\calS_F$ or $\calS_{f^*}$ are distinct primes. So $B$ depends only on $F$ and the choices of $\mathcal{S}_F$ and $\calS_{f^*}$; and hence depends only on $F$. In the definition of $P_F$ we take $C_F$ large enough so that
\[
Q_F\prod_{f\in\calS_F}Q_{f^*}\bigg|\prod_{p\le C_F}p=P_F.
\]
By Proposition \ref{bqf2}, we can bound $\calB_1(M, N)$ by
\begin{equation} \label{BL1} \calB_1(M,N) \leq \sum_{f\in\mathcal{S}_F} \sum_{\substack{M < f(u, v) \leq 2M\\\gcd(f(u, v), P_F)=1\\\gcd(u, v)=1}} \bigg|\sum_{\substack{N < f^*(w, z) \leq N' \\ \gcd(f(u, v)P_F, f^*(w, z)) = 1\\\gcd(w, z)=1 }} \alpha(f^*(w, z)) \lam(\calQ_F(u, v; w, z)) \bigg|.\end{equation}
Proceeding with the argument to relax the condition that $\gcd(f(u, v), f^*(w, z)) = 1$, we use the familiar arithmetic identity
\[\sum_{r | \gcd(m,n)} \mu(r) = \begin{cases} 1 & \text{if } \gcd(m,n) = 1 \\ 0 & \text{otherwise}.\end{cases}\]
Since $n$ is squarefree, by Proposition \ref{bqf1} we can decompose $f^*(w, z)=n$ as
\begin{equation}\label{ggstar}
g(u_0,v_0) = r, g^*(w_0, z_0) = \frac{n}{r}
\end{equation}
for some $g\in\mathcal{S}_{f^*}$ and we have the relations
\begin{equation}\label{pair}
\begin{split}
\bigg(du_0+\frac{e+B}{2}v_0\bigg)w_0+\bigg(\frac{(B+e)B+\Delta-Be}{4da\alpha}v_0\bigg)z_0&=w,\\
-a\alpha v_0w_0+\bigg(u_0-\frac{B-e}{2d}v_0\bigg)z_0&=z.
\end{split}
\end{equation}
We then see that the inner sum of (\ref{BL1}) becomes 
\[\sum_{g\in\mathcal{S}_{f^*}} \sum_{g(u_0, v_0) = r} \mu(r) \sum_{\substack{N < r g^*(w_0, z_0) < N'\\\gcd(g^*(w_0, z_0), P_F)=1\\\gcd(w_0, z_0)=1}} \alpha(r g^*(w_0, z_0)) \lam(\calQ_F(u, v; w, z)) .    \]
Now it suffices to evaluate a sum of the shape 
\begin{equation} \label{BL2}
{\sum_{r}}^\flat \sum_{g(u_0, v_0) = r} \sum_{\substack{M < f(u, v) \leq 2M\\\gcd(f(u, v), P_F)=1\\\gcd(u, v)=1\\r|f(u, v)}}\bigg|  \sum_{\substack{N < r g^*(w_0, z_0) < N'\\\gcd(g^*(w_0, z_0), P_F)=1\\\gcd(w_0, z_0)=1}} \alpha(r g^*(w_0, z_0)) \lam(\calQ_F(u, v; w, z)) \bigg|
\end{equation}
where $w$ and $z$ are determined by \eqref{pair}. Now note that
\[
f(u, v)g(u_0, v_0)=adP^2+BPQ+\frac{B^2+\Delta}{4ad}Q^2:=h(P, Q)
\]
where
\begin{equation}
\label{pair2}
\begin{split}
P&=\bigg(u-\frac{B-b}{2a}v\bigg)u_0-\bigg(\frac{B-e}{2d}u+\frac{(b+e)B+\Delta-be}{4ad}v\bigg)v_0,\\
Q&=dvu_0+\bigg(au+\frac{b+e}{2}v\bigg)v_0.
\end{split}
\end{equation}
When $(P, Q)$ and $(u_0, v_0)$ are fixed, there is at most one pair $(u, v)$ such that \eqref{pair2} holds. Also with \eqref{pair} and \eqref{pair2} we deduce that
\[
-\alpha vw+\bigg(u-\frac{B-b}{2a}v\bigg)z = z_0P-\alpha w_0Q.
\]
By \eqref{uBv} we have $\gcd(P, \alpha)=1$. Therefore the sum \eqref{BL2} is less than
\[
{\sum_{r}}^\flat \rho(r)\sum_{\substack{\substack{M < h(P, Q) \leq 2M\\\gcd(P, \alpha)=1}\\\gcd(h(P, Q), P_F)=1\\r^2|h(P, Q)}}\bigg| \sum_{\substack{N < r g^*(w_0, z_0) < N'\\\gcd(g^*(w_0, z_0), P_F)=1\\\gcd(w_0, z_0)=1}} \alpha(r g^*(w_0, z_0))\lambda(z_0P-\alpha w_0Q) \bigg|.
\]
Estimating trivially we find that the terms with $r \geq \theta^{-2}$, where we take $\theta = (\log x)^{-A}$ for some large positive number $A$ as in \cite{FI}, contribute
\[O \left(\theta MN \sum_{r >\theta^{-2}} \rho(r)^2 r^{-2} \right) = O \left(\theta^3 x (\log x)^2\right).\]
In the remaining terms we ignore the conditions $r^2 | h(P, Q)$, $\gcd(h(P, Q), P_F)=1$ and obtain 
\begin{align*} \calB_1(M,N) & \leq \sum_{r < \theta^{-2}} \rho(r) \sum_{\substack{rM < h(P, Q) \leq 2rM\\\gcd(P, \alpha)=1}} \bigg|  \sum_{\substack{N < r g^*(w_0, z_0) < N'\\\gcd(g^*(w_0, z_0), P_F)=1\\\gcd(w_0, z_0)=1}} \alpha(r g^*(w_0, z_0))\lambda(z_0P-\alpha w_0Q))\bigg| \\
& + O \left(\theta^3 x (\log x)^2\right).
\end{align*}

Put
\[\calC_r(M,N) = \sum_{\substack{M < h(P, Q) \leq 2M\\\gcd(P, \alpha)=1}} \bigg|  \sum_{\substack{N < r g^*(w_0, z_0) < N'\\\gcd(g^*(w_0, z_0), P_F)=1\\\gcd(w_0, z_0)=1}} \alpha(r g^*(w_0, z_0))\lambda(z_0P-\alpha w_0Q) \bigg|.\]
We then write
\[\calC_{cr} (M,N) = \sideset{}{^\ast} \sum_{\substack{M < h(P, Q) \leq 2M\\\gcd(P, \alpha)=1}} \bigg|\sum_{\substack{N < r g^*(w_0, z_0) < N'\\\gcd(g^*(w_0, z_0), P_F)=1\\\gcd(w_0, z_0)=1}} \alpha(r g^*(w_0, z_0))\lambda(z_0P-\alpha w_0Q) \bigg|,\]
where the asterisk in the sum means that the sum is over primitive pairs. By \cite{FI} it then suffices to give a bound of the shape
\[\calC_{cr}(M,N) \ll \theta^5 MN\]
for every $c,r, M,N$ with $c < \theta^{-4}, r < \theta^{-2}, M\ge \theta^4Z, N > \theta^3 Y,$ and $\theta^5 X < MN < X$. Our assumptions in Lemma \ref{Blem} guarantee that $M, N$ satisfy $N^\eps<M<N^{1-\eps}$ for some small $\eps>0$. This assumption will be used in \eqref{Xi} and \eqref{pow} and we will give a bound of the form
\begin{equation} \label{req bd} \calC_{cr}(M,N) \ll MN (\log N)^{-j}.\end{equation}

Let $A=\sqrt{N/r}, B=\sqrt{M}$ and $\alpha(u, v)=\alpha(rg^*(u, v))$. Then $\alpha(u, v)$ is supported in the annulus $A^2< g^*(u,v) \leq 4A^2$. By applying the Cauchy-Schwarz inequality, we obtain
\begin{align*} |\calB(M,N)| & \leq \sum_{\ell} |\lambda(\ell)| \sideset{}{^\ast} \sum_{\substack{M<h(w, z)<2M\\\gcd(w, \alpha)=1}}\bigg| \sum_{vw-\alpha uz = \ell} \alpha(u, v) \bigg| \\
& \leq A^{1/2}B^{3/2} \calD(\alpha)^{1/2},
\end{align*}
where 
\begin{equation}\label{calDa}
\calD(\alpha) = \sideset{}{^\ast} \sum_{\substack{(w, z)\in\mathbb{Z}^2\\\gcd(w, \alpha)=1}} \psi (w, z) \sum_{\ell} \bigg| \sum_{Q(u, v; w, z) = \ell} \alpha(u, v) \bigg|^2
\end{equation}
and
\[
Q(u, v; w, z)=vw-\alpha uz.
\]
Here $\psi(w, z)$ can be any non-negative function with $\psi(w, z) \geq 1$ if $B^2 \leq h(w, z) \leq 4B^2$. We do not need to be specific at this point; nevertheless it will be convenient to assume that $\psi(w, z)$ takes the form $\Psi(h(w, z))$, where
\[0 \leq \Psi(t) \leq 1, \Psi(t) = 1 \text{ if } B^2 \leq t \leq 4B^2, \]
\[\operatorname{supp} \Psi \subset [B^2/4, 9B^2], \Psi^{(j)} \ll B^{-2j}.\]
Our desired estimate for $\calD(\alpha)$ is $A^3B$ with a saving of an arbitrary power of $\log N$. Since $\ell$ runs over all integers (without any restriction), after squaring we obtain
\begin{equation} \label{Da} \calD(\alpha) = \sideset{}{^\ast} \sum_{\substack{(w, z)\in\mathbb{Z}^2\\\gcd(w, \alpha)=1}} \psi(w, z) \sum_{Q(u, v; w, z) = 0} (\alpha \ast \alpha)(u, v),
\end{equation}
where
\[(\alpha \ast \alpha)(u, v) = \sum_{(s_1, t_1) - (s_2, t_2) = (u, v)} \alpha(s_1, t_1) \overline{\alpha}(s_2, t_2). \]
This equality follows because $Q(u, v; w, z)$ is a bilinear form. Note that
\[(\alpha \ast \alpha)(0,0) \ll A^2.\]
The orthogonality relation $Q(u, v; w, z) = 0$ in (\ref{Da}) is equivalent to
\[(u, v) = (cw, c\alpha z) \]
for some rational integer $c \in \Z$ since $ \gcd(w, \alpha z) = 1$.
It thus follows that 
\begin{equation}\label{D}
\begin{split}
\calD(\alpha) = &\sum_{c \in \Z} \sideset{}{^\ast} \sum_{\substack{(w, z)\in\mathbb{Z}^2\\\gcd(w, \alpha)=1}} \psi(w, z) (\alpha \ast \alpha)(cw, c\alpha z)\\
 = &\calD_0(\alpha) + 2 \calD^\ast(\alpha),
\end{split}
\end{equation}
say, where $\calD_0(\alpha)$ denotes the contribution of $c = 0$ and $\calD^\ast(\alpha)$ that of all $|c| > 0$. Thus
\begin{equation}\label{d0} \calD_0(\alpha) = \lVert \alpha \rVert^2 \sideset{}{^\ast} \sum_{\substack{(w, z)\in\mathbb{Z}^2\\\gcd(w, \alpha)=1}} \psi(w, z) \ll A^2B^2
\end{equation} 
and
\begin{equation} \calD^\ast(\alpha) = \sum_{(s, t) \ne (0, 0)} \psi\left(\frac{s}{ \gcd(s, t)}, \frac{t}{\gcd(s,  t)} \right) (\alpha \ast \alpha)(s, \alpha t).
\end{equation}
We trade the primitivity condition for congruence conditions by means of M{\"o}bius inversion, getting
\begin{equation} \calD^\ast(\alpha) = \sum_{b,c > 0} \mu(b) \calD(\alpha; b, c)
\end{equation} 
where
\begin{equation} \calD(\alpha; b, c) = \sum_{ (s, t) \equiv (0,0) \imod{bc}} \psi \bigg(\frac{s}{ c}, \frac{t}{c}\bigg) (\alpha \ast \alpha)(s, \alpha t).
\end{equation}
Note that $g^*(s, \alpha t) \leq 2A$ (from the support of $\alpha$) and $cB/2 < g^*(s, \alpha t) < 3cB$ (from the support of $\psi$). Observe that these imply that $c < 4AB^{-1}$, otherwise $\calD(\alpha; b, c)$ is zero. Let $\Xi$ be a parameter such that
\begin{equation}\label{Xi} 1 \leq \Xi \leq 4AB^{-1} = C,
\end{equation}
say. We will take $\Xi$ to be a power of $\log N$ at the end and this explains why $N$ needs to be larger than $M$, say $N^{1-\epsilon}>M$. By the trivial bound
\[\calD(\alpha; b, c) \ll A^2B^2 b^{-2} \]
we see that the terms with $b \geq \Xi$ or $c \leq C\Xi^{-1}$ contribute at most $O(A^3B \Xi^{-1})$ to $\calD^\ast(\alpha)$ so 
\begin{equation}\label{Dast}
\calD^\ast(\alpha) = \sum_{b \leq \Xi} \mu(b) \sum_{C\Xi^{-1} < c < C} \calD(\alpha; b, c) + O \left(A^3B\Xi^{-1}\right).
\end{equation}
If $h(w, z)=Dw^2+Ewz+Fz^2$ with $D>0$, then
\[
\begin{split}
\psi(w, z)&=\Psi\bigg(\frac{(2Dw+Ez)^2+(4DF-E^2)z^2}{4D}\bigg)\\
&=\Psi\bigg(\bigg(\frac{2Dw+Ez}{2\sqrt{D}}\bigg)^2+\bigg(\frac{\sqrt{|\Delta|}z}{2\sqrt{D}}\bigg)^2\bigg).
\end{split}
\]
Then we can define
\[
\psi_0(w, z)=\psi\bigg(\frac{\sqrt{|\Delta|}w-Ez}{\sqrt{D|\Delta|}}, \frac{2\sqrt{D}z}{\sqrt{\Delta}}\bigg).
\]
Then
\[
\psi(w, z)=\psi_0\bigg(\frac{2Dw+Ez}{2\sqrt{D}}, \frac{\sqrt{|\Delta|}z}{2\sqrt{D}}\bigg)\hspace{5mm}\text{and}\hspace{5mm}\psi_0(w, z)=\Psi(w^2+z^2).
\]
Hence if we define
\[
\begin{split}
\phi(x, y)&:=\int^\infty_{-\infty}\int^\infty_{-\infty}\psi_0(w, z)e(-(xw+yz))\, dw\, dz\\
&=\int^\infty_{-\infty}\int^\infty_{-\infty}\psi\bigg(\frac{2Dw+Ez}{2\sqrt{D}}, \frac{\sqrt{|\Delta|}z}{2\sqrt{D}}\bigg)e(-(xw+yz))\, dw\, dz,	\\
\end{split}
\]
$\phi(x, y)$ will depend only on $x^2+y^2$ and we can set $\phi(x, y)=\Phi(x^2+y^2)$. By inversion and a change of variables we obtain
\begin{equation}\label{integral}
\psi\bigg(\frac{w}{c}, \frac{z}{c}\bigg)=\frac{2c^2}{\sqrt{|\Delta|}}\int^\infty_{-\infty}\int^\infty_{-\infty}\Phi\bigg(\frac{4c^2h(-y, x)}{|\Delta|}\bigg)e(xw+yz)\, dx\, dy.
\end{equation}
Therefore
\[
\calD(\alpha; b, c) = \frac{2c^2}{\sqrt{|\Delta|}}\int^\infty_{-\infty}\int^\infty_{-\infty}\Phi\bigg(\frac{4c^2h(-y, x)}{|\Delta|}\bigg)S_{bc}(x, y)\, dx\, dy
\]
where
\[
\begin{split}
S_{d}(x, y) &= \sum_{ (s, t) \equiv (0,0) \imod d}(\alpha \ast \alpha)(s, \alpha t)e(xs+yt)\\
&= \mathop{\sum\sum}_{\substack{s_1\equiv s_2\imod d\\t_1\equiv t_2\imod {\lcm(\alpha, d)}}}\alpha(s_1, t_1)\overline{\alpha}(s_2, t_2)e\bigg(xs_1+\frac{yt_1}{\alpha}\bigg)e\bigg(-xs_2-\frac{yt_2}{\alpha}\bigg)\\
&= \sum_{\alpha_1\imod \alpha}\sum_{d_1, d_2\imod d}\bigg|\sum_{\substack{s\equiv d_1\imod d\\ t\equiv d_2\imod d\\t\equiv \alpha_1\imod \alpha}}\alpha(s, t)e\bigg(xs+\frac{yt}{\alpha}\bigg)\bigg|^2.
\end{split}
\]
By (9.14) of \cite{FI},
\[
c^2\Phi\bigg(\frac{4c^2h(-y, x)}{|\Delta|}\bigg) \ll \frac{c^2B^2}{(1+c^2B^2h(-y, x))^{3/2}} \ll \frac{A^2\Xi}{(1+h(-y, x)A^2)^{3/2}}.
\]
Hence
\[
\calD(\alpha; b, c) \ll \Xi A^2\int^\infty_{-\infty}\int^\infty_{-\infty}H(x, y)S_{bc}(x, y)\, dx\, dy
\]
where
\[
H(x, y) = \frac{1}{(1+h(-y, x)A^2)^{3/2}}.
\]
By grouping $d=bc$ and setting $D=C\Xi$, we obtain from \eqref{Dast}
\begin{equation}\label{dstar}
\calD^*(\alpha)\ll M\Xi^3\int^\infty_{-\infty}\int^\infty_{-\infty}H(x, y)\bigg(\sum_{d\le D}d^2S_{d}(x, y)\bigg)\, dx\, dy+A^3B\Xi^{-1}.
\end{equation}
To account for the large $d$ appearing in the above sum, we need to invoke Proposition 15 of \cite{FI}.
\begin{proposition}
Suppose $A\ge D\ge 1$. Let $f$ be a complex-valued function on $\Z[i]$ supported on the disc $|z|\le A$. Define
\[
S_f(D)=\sum_{d\le D}d^2\sum_{\delta\imod d}\bigg|\sum_{z\equiv\delta\imod d}f(z)\bigg|^2.
\]
Then for any $G\ge1$ we have
\begin{equation}
\label{pow}
S_f(D)\le 2DS_f(G)+O_\eps(AD(D^{1+\epsilon}+AG^{\epsilon-1})||f||^2).
\end{equation}
\end{proposition}
For $m+ni\in\Z[i]$, we take $f(m+ni)=\alpha(m, n)e(xm+yn/\alpha)$ if $n\equiv \alpha_1\imod \alpha$. Thus
\begin{equation}\label{dSd}
\sum_{d\le D}d^2S_d(x, y)\le 2\alpha D\sum_{d\le G}d^2S_d(x, y)+O(A^5B^{-1}\Xi G^{\epsilon-1})
\end{equation}
where
\[
\calD_d(\alpha) = \int^\infty_{-\infty}\int^\infty_{-\infty}H(x, y)S_d(x, y)\, dx\, dy.
\]
Similar to $\Xi$, we expect $G$ is a power of $\log N$. To apply \eqref{pow} we need $DG<A^{1-\eps}$, which is valid if $B>A^\eps$. By taking $G=\Xi^6$ and substituting \eqref{d0}, \eqref{dstar} and \eqref{dSd} into \eqref{D}, we arrive at
\[
\calD(\alpha)\ll AB\Xi^4\sum_{d\le \Xi^6}d^2\calD_d(\alpha)+A^2(B^2+AB\Xi^{-1})
\]
where
\[
\begin{split}
\calD_d(\alpha) &= \int^\infty_{-\infty}\int^\infty_{-\infty}H(x, y)S_d(x, y)\, dx\, dy\\
&=  \sum_{ (s, t) \equiv (0,0) \imod d}(\alpha \ast \alpha)(s, \alpha t) \int^\infty_{-\infty}\int^\infty_{-\infty}H(x, y)e(xs+yt)\, dx\, dy.
\end{split}
\]

Our final obstacle is to develop an estimate of $\calD_d(\alpha)$ for small values of $d$. Here the modulus $d$ is less than a power of $\log N$, which is analogous to the classical Siegel-Walfisz theorem. As in \eqref{integral}, after some changes of variables the above integral can be expressed as
\[
\int^\infty_{-\infty}\int^\infty_{-\infty}H(x, y)e(xw+yz)\, dx\, dy= \frac{2\pi}{A^2}\exp\bigg(-\frac{4\pi\sqrt{h(w, z)}}{A\sqrt{|\Delta|}}\bigg).
\]
Hence
\[
\calD_d(\alpha)=2\pi A^{-2} \mathop{\sum\sum}_{\substack{s_1\equiv s_2\imod d\\t_1\equiv t_2\imod {\lcm(\alpha, d)}}}\alpha(s_1,  t_1)\overline{\alpha}(s_2, t_2)\exp\bigg(-\frac{4\pi\sqrt{h(s_1, t_1; s_2, t_2)}}{A\sqrt{|\Delta|}}\bigg)
\]
where
\[
h(s_1, t_1; s_2, t_2)=h(s_1-s_2, \alpha^{-1}(t_1-t_2)).
\]
Note that
\[
\begin{split}
\calD_d(\alpha)\ll&\max_{d_1, d_2\imod {\alpha d}}\max_{N<g^*(s_0, t_0)<N'}\bigg|\sum_{\substack{(s, t)\equiv (d_1, d_2)\imod {\alpha d}\\N<g^*(s, t)<N'}}\\
&\mu(g^*(s,  t))\chi(g^*(s,t))\exp\bigg(-\frac{4\pi\sqrt{h(s, t; s_0, t_0)}}{A\sqrt{|\Delta|}}\bigg)\bigg|.
\end{split}
\]
Hence it suffices to show that
\[
\sum_{\substack{(s, t)\equiv (d_1, d_2)\imod {\alpha d}\\N<g^*(s, t)<N'}}\mu(rg^*(s, t))\chi(rg^*(s, t))\exp\bigg(-\frac{4\pi\sqrt{h(s, t; s_0, t_0)}}{A\sqrt{|\Delta|}}\bigg) \ll N\eta.
\]
Define $\eta=(\log N)^{-j}$. We can divide the region $N<g^*(s, t)<N'$ into non-overlapping sectors of the form
\[
R(Z, \xi)=\{(s, t)\in\Z^2:Z-\sqrt{N}\eta<g^*(s, t)\le Z, \xi<\arg(s+ti)\le \xi+\eta\}
\]
and there are at most $\eta^{-2}$ regions. For a fixed $(S, T)\in R(Z, \xi)$ and any $(s, t)\in R(Z, \xi)$, we always have
\[
\exp\bigg(-\frac{4\pi\sqrt{h(s, t; s_0, t_0)}}{A\sqrt{|\Delta|}}\bigg)=\exp\bigg(-\frac{4\pi\sqrt{h(S, T; s_0, t_0)}}{A\sqrt{|\Delta|}}\bigg)+O(\eta).
\]
Hence it suffices to show that
\[
\sum_{\substack{(s, t)\in R(Z, \xi)\\(s,  t)\equiv (d_1, d_2)\imod {\alpha d}\\\gcd(g^*(s, t), r)=1}}\mu(g^*(s, t))\chi(g^*(s, t)) \ll N\eta^3.
\]
This is a special case of Lemma 3.3.6 of \cite{Helf}.
\begin{lemma}
Let $Q(x, y)$ be a primitive positive definite quadratic form. Let $H\le (\log X)^N$. Then for any $h_1, h_2\imod H$, any $A>0$ and sector $S\subset\R^2$,
\[
\sum_{\substack{Q(x, y)\le X\\x\equiv h_1\imod H\\y\equiv h_2\imod H\\(x, y)\in S}}\mu(Q(x, y))\ll_{A, N} X(\log X)^{-A}.
\]
\end{lemma}
\noindent Helfgott proved this for the Liouville function $\lam$ but the same proof also works for $\mu$. This concludes our proof of Lemma \ref{Blem}.

\bibliographystyle{amsbracket}
\providecommand{\bysame}{\leavevmode\hbox to3em{\hrulefill}\thinspace}

\end{document}